\documentclass[letterpaper,11pt]{amsart}

\usepackage[T1]{fontenc}
\usepackage[utf8]{inputenc}
\usepackage[english]{babel}
\usepackage{a4wide}
\usepackage{amsfonts}
\usepackage{amssymb,,amsmath}
\usepackage[dvipsnames]{xcolor}
\usepackage{pgfplots}
\usepackage{graphicx}
\usepackage{stmaryrd}
\usepackage{mathtools}
\mathtoolsset{showonlyrefs}
\usepackage{bbm} 
\usepackage{amsthm}
\usepackage{float} 
\usepackage{hyperref}
\usepackage{comment}
\usepackage{enumerate}
\hypersetup{
	colorlinks=true,       
	linkcolor=Blue,          
	citecolor=Blue,        
	filecolor=magenta,      
	urlcolor=cyan           
}

\theoremstyle{plain}
\newtheorem{theorem}{Theorem}[section]
\newtheorem{proposition}[theorem]{Proposition}
\newtheorem{definition}[theorem]{Definition}
\newtheorem{lemma}[theorem]{Lemma}

\numberwithin{equation}{section}

\newtheorem{remark}[theorem]{Remark}

\usepackage{tikz}
\usetikzlibrary{calc}

\DeclareMathOperator{\Ran}{Ran}

\DeclareMathOperator{\Vol}{vol}
\usepackage{dsfont}

\newcommand{\R}{\mathbb{R}}

\newcommand{\dd}{{\rm d}}
\newcommand{\del}{\partial}

\newcommand{\de}{\, \mathrm{d}}

\newcommand{\set}[1]{\left\{ #1 \right\}}
\newcommand{\norm}[1]{\left\| #1 \right\|}

\DeclareMathOperator{\Lip}{Lip}
\newcommand{\blip}{\Lip}

\renewcommand{\le}{\leqslant}
\renewcommand{\ge}{\geqslant}

\renewcommand{\tilde}{\widetilde}
\newcommand{\eps}{\varepsilon}
\renewcommand{\epsilon}{\varepsilon}
\renewcommand{\phi}{\varphi}

\newcommand{\N}{\mathbb N}

\newcommand{\CH}{\mathcal{H}}

\newcommand{\CM}{\mathcal{M}}

\newcommand{\BG}{\mathbf{G}}

\newcommand{\BY}{\mathbf{Y}}

\newcommand{\bo}\boldsymbol{}

\newcommand{\RC}{\mathrm{C}}

\newcommand{\RL}{\mathrm{L}}

\newcommand{\RW}{\mathrm{W}}

\newcommand{\rd}{\mathrm{d}}

\title[Spectral inequalities in manifolds with bounded curvature]{Metric-uniform
spectral inequality for the Laplacian on manifolds with bounded sectional curvature}
\author{Alix Deleporte}
\address{Laboratoire de Math\'ematiques d'Orsay, Universit\'e Paris-Saclay,
CNRS, B\^atiment~307, 91405 Orsay Cedex, France}
\email{alix.deleporte@universite-paris-saclay.fr}
\author{Jean Lagac\'e}
\address{Department of Mathematics, King's College London,
Strand, London, WC2R 2LS, UK}
\email{jean.lagace@kcl.ac.uk}
\author{Marc Rouveyrol}
\address{Fakultät für Mathematik, Universität Bielefeld, Universitätsstrasse 25,
33615 Bielefeld, Germany}
\email{marc.rouveyrol@universite-paris-saclay.fr}

\newcommand{\vide}[1]{}

\begin{document}
	
	\begin{abstract}
		Given a Riemannian manifold $M$ endowed with a smooth metric $g$
        satisfying upper and lower sectional curvature bounds, we show an
        equivalence property between the $\RL^2$ norm on $M$ and the $\RL^2$ norm on
        subsets $\omega$ satisfying a thickness condition, for functions in the
        range of a spectral projector. The thickness condition is known to be
        optimal in this setting. The constant appearing in the equivalence of
        norms property depends only on the dimension of the manifold, curvature
        bounds, and frequency threshold of the spectral cutoff, but, crucially,
        not on the injectivity radius.
	\end{abstract}
	
	\maketitle
	
	\section{Introduction}

\subsection{Metric-uniform spectral inequality}

Let $(M,g)$ be smooth, complete Riemannian manifold. We investigate quantitative
unique continuation from a measurable set $\omega \subset M$ for functions with compact (Laplace)-spectral support, a
type of result known as a \emph{spectral inequality} in the control
theory literature.
Starting with the seminal work of  Lebeau and Robbiano
\cite{lebeaurobbiano1995controlechaleur}, spectral inequalities are related to
null-controllability (or equivalently, final-time observability) of the heat
equation. 

When $M$ has Ricci curvature bounded below, Deleporte and Rouveyrol proved in \cite{deleporterouveyrol2024spectralhypsurfaces}
that a necessary condition for $\omega$ to support a spectral inequality is that
it satisfies some quantitative equidistribution property in the whole manifold, known as
\emph{thickness}. In the compact
\cite{burqmoyano2022propagationheat}, flat \cite{egidiveselic2018controlheat,wangwangzhangzhang2019observability}
and hyperbolic
\cite{rouveyrol2024spectralH2,deleporterouveyrol2024spectralhypsurfaces}
settings, thickness of $\omega$ has in turn been
shown to imply that it supports a spectral inequality. In this context, the aim of the present article is to prove that the spectral
inequality holds from thick subsets in any manifold with bounded
sectional curvature, with bounds that depend only on the
dimension, the norm of the sectional curvature and the thickness of $\omega$
but, crucially, not on the injectivity radius of $M$ (which may tend to
zero at infinity). This result implies null-controllability of the heat equation with constants depending only
on those geometric quantities.

In order to state our results precisely, let us define a few objects. On
$(M,g)$, we denote by $- \Delta_g \ge 0$ the Laplacian, also known as the Laplace--Beltrami operator, defined
as the Friedrichs extension of the Dirichlet form, see \cite{lewin2022livrespectral}. Since
the Laplacian is non-negative and self-adjoint it has a
positive square root, to which we associate a spectral measure $\mu$ supported on $[0,\infty)$ and forming a resolution of the
identity in the sense that for all $u \in \RL^2(M)$, $u = \int_0^\infty \de
\mu(\lambda) u$. For $\Lambda > 0$, let $\Pi_\Lambda$ be the projector 
 onto the spectral window
$[0,\Lambda]$ defined via the functional calculus as
\begin{equation}
    \Pi_\Lambda u = \int_0^{\Lambda} \de \mu(\lambda) u.
\end{equation}
\begin{definition}
 Let $(M,g)$ be a Riemannian manifold and $\omega \subset M$ be measurable. We
 say that $(M,g,\omega)$ \emph{satisfies the spectral inequality} if there
 exists $C_{\mathrm{spec}} > 0$ such that for all $\Lambda > 0$ and $u \in
 \RL^2(M)$,
	\begin{equation}
		\label{eq: spec} \tag{Spec}
        \|\Pi_\Lambda u\|_{\RL^2(M)} \leq e^{C_{\mathrm{spec}}(\Lambda + 1)}
        \|u\|_{\RL^2(\omega)}.
	\end{equation}
\end{definition}
Thickness of a measurable set $\omega$ is, on its face, a purely measure
theoretical property.

\begin{definition}
	\label{def:thickness}
	Let $(M,g)$ be a Riemannian manifold, $\Vol$ its canonical Riemannian measure, $\omega \subset M$ measurable, and $R > 0$, $\epsilon \in (0,1)$. 
    We say that $\omega$ is \emph{thick at scale $R$ with ratio $\epsilon$} (or
    $(R, \epsilon)$-thick) if for all $x \in M$, denoting $B_x(R) \subset M$ the geodesic ball of center $x$ and radius $R$,
	\begin{equation}
		\label{eq: thickness} \tag{Thick}
		\frac{\Vol(\omega \cap B_x(R))}{\Vol(B_x(R))} \geq \epsilon.
	\end{equation}
	We say that $\omega$ is \emph{thick} if there exist some $R > 0$ and
    $\epsilon \in (0,1)$ such that $\omega$ is thick at scale $R$ with ratio
    $\epsilon$. In the rest of the article, $(R, \epsilon)$ will be called the
    \emph{thickness parameters}.
\end{definition}
We aim at addressing the following questions.
\begin{enumerate}[1.]
	\item \label{item: generality question} What is the largest class of manifolds on which the equivalence between \eqref{eq: spec} and \eqref{eq: thickness} holds?
	\item \label{item: uniformity question} Can the constant $C_{\rm spec}$ be
        proved to depend only on geometric parameters of the manifolds $(M,g)$
        (along with the thickness parameters $(R,\eps)$) rather than on the metric $g$ itself?
\end{enumerate}
With respect to Question \ref{item: generality question}, we obtain a (partial)
converse to the necessity statement of
\cite{deleporterouveyrol2024spectralhypsurfaces}, while for Question
\ref{item: uniformity question} we prove that the dimension and a bound on the
sectional curvature are the only geometric quantities on which the spectral
inequality depends. For $d \in \N$ and $\kappa \ge 0$, denote by $\CM(d,\kappa)$
the set of $d$-dimensional manifolds whose sectional curvature has norm smaller
than $\kappa$.
\begin{theorem}
	\label{theo: main result}
    Let $d \geq 2$, $\kappa \ge 0$ and $R, \epsilon > 0$. 
    There exists a constant $C_{\rm spec}$ such that for any
    $(R,\eps)$-thick subset $\omega$ of a  Riemannian manifold $(M,g) \in \mathcal{M}(d,\kappa)$ 
supports a spectral inequality with constant $C_{\rm spec}$.
\end{theorem}

\subsection{Applications to observability and control theory}
One application of spectral inequalities of the form \eqref{eq:
  spec} is the controllability and observability of the heat equation. 
\begin{definition}
	The \emph{controllability cost} at time $T > 0$ for the heat equation is the constant
	\[
	C_{\rm cont}(T,M,\omega) = \sup_{u_0, f} \frac{\|f\|_{\RL^2((0,T) \times \omega)}}{\| u_0\|_{\RL^2(M)}},
	\]
	where the supremum is taken over all initial conditions $u_0 \in \RL^2(M) \setminus \{0\}$ and controls $f~\in~\RL^2((0,T) \times M)$ such that the solution of
	\begin{equation}
		\label{eq: heat}
		\left\{ \begin{array}{rll}
			\partial_t u - \Delta_g u &= \mathbbm{1}_\omega f & \text{ in } \mathbb{R}^+_t \times M,\\
			u_{|t=0} &= u_0 & \text{ in } M
		\end{array} \right.
	\end{equation}
	satisfies $u_{|t=T} = 0$.
	
	The \emph{obervability cost} for the heat semigroup at time $T > 0$ is the smallest constant $C_{\rm obs}(T,M,\omega)$ such that the final-time observability inequality
	\begin{equation}
		\label{eq: obs} \tag{Obs}
		\forall u \in \RL^2(M), \|e^{T\Delta_g} u\|_{\RL^2(M)} \leq C_{\rm obs} \int_0^T \int_\omega |e^{t\Delta_g} u(x)|^2 d\Vol dt
	\end{equation}
	is satisfied.
	\end{definition}
	By duality, controllability and observability are equivalent;
    moreover, using general functional calculus arguments, both are
    consequences of the spectral inequality \eqref{eq: spec}.
	\begin{proposition}( \cite[Theorem
          2.44]{coron2007controlnonlinearity},
          \cite[Section 3.2]{miller2010directlebeaurobbiano},
          \cite{wangwangzhangzhang2019observability})
          \label{prop: control}
		Let $(M,g)$ be a complete Riemannian manifold, and $\omega \subset M$ a measurable subset of $M$ such that \eqref{eq: spec} holds from $\omega$. Then, $\omega$ satisfies \eqref{eq: obs} and there exists $c_1 > 0$ such that for every $T > 0$,
		\[
		C_{\rm cont}(T)=C_{\rm obs}(T) \leq \exp(c_1(1 + C_{\rm spec})^2(1 + \tfrac{1}{T})).
		\]
	\end{proposition}
In the context of Theorem \ref{theo: main result}, given $d, R,\epsilon,\kappa$, Proposition \ref{prop: control} yields a uniform observability cost for $(R,\epsilon)$-thick sets on $d$-dimensional manifolds with curvature
bounded by $\kappa$.

Deleporte and Rouveyrol proved the following converse statement, which
illustrates the sharpness of Theorem \ref{theo: main result}.
\begin{proposition}[\cite{deleporterouveyrol2024spectralhypsurfaces}, Theorem 2]
	\label{theo: converse}
	Let $d \geq 2$, $(M,g)$ be a $d$-dimensional manifold with Ricci curvature bounded from below and $\omega \subset M$. If $C_{\rm obs}(T,M,\omega)<+\infty$, then $\omega$ is thick.
\end{proposition}

Beyond \eqref{eq: obs}, we also prove an observability result for the heat equation \emph{at
  fixed time}. 
\begin{theorem}\label{theo: fixed time obs}
    Let Let $d \geq 2$, $\kappa \ge 0$, and $R, \eps, s > 0$. There exist $C>0$ and $\alpha\in (0,1)$
    such that for any Riemannian manifold $(M,g) \in
    \mathcal{M}(d,\kappa)$, and any $(R,\epsilon)$-thick set $\omega
    \subset M$, for every $u\in \RL^2(M,\R)$,
    \[
      \|e^{s\Delta_g}u\|_{\RL^2(M)}\leq C\|e^{s\Delta_g}u\|_{\RL^2(\omega)}^{\alpha}\|u\|_{\RL^2(M)}^{1-\alpha}.
    \]
    In particular, there exists $C_1$ such that, for every $\delta>0$,
    \[
      \|e^{s\Delta_g}u\|_{\RL^2(M)}\leq
      C_1\delta^{-\frac{1-\alpha}{\alpha}}\|e^{s\Delta_g}u\|_{\RL^2(\omega)}+
      \delta \|u\|_{\RL^2(M)}.
      \]
    \end{theorem}

\subsection{Further remarks and comparison to existing results}

The spectral inequality \eqref{eq: spec} has previously been investigated
in the non-compact manifold setting
in \cite{miller2005uniquecontinuationnoncompact} and
\cite{rosetautenhahn2023unique}. Our geometric assumption is slightly stronger
than the one used in both of these works (bounded sectional curvature rather
than a lower bound on Ricci curvature), though we do not need the interpolation
inequality assumption of \cite{miller2005uniquecontinuationnoncompact} and
obtain a high-frequency result, unlike in \cite{rosetautenhahn2023unique}. The
sufficiency of thickness for spectral observability is new in such general
geometries, and so is the uniform character of $C_{\rm spec}$. Besides, our
proof is still valid in the compact setting, generalizing previous spectral
estimate results
\cite{burqmoyano2022propagationheat,laurentleautaud2021conjlucmiller} by
specifying the qualitative geometric dependence of the constant.

Theorem \ref{theo: fixed time obs} is inspired by similar pointwise observability results for weakly diffusive \cite{alphonsemartin2022nullcontdiffeq} and Baouendi--Grushin operators \cite{alphonseseelmann2025baouendigrushin}, which stem from similar propagation of smallness arguments.

\subsection{Summary of the proof}
\label{sec: sketch of proof}

We now sketch the proof of Theorems \ref{theo: main result} and \ref{theo: fixed time obs}. The general scheme borrows from the works of Jerison, Lebeau and Robbiano \cite{lebeaurobbiano1995controlechaleur,jerisonlebeau1999nodal} and their recent adaptations to the non-compact manifold setting \cite{burqmoyano2021spectralestimates,rouveyrol2024spectralH2,deleporterouveyrol2024spectralhypsurfaces}.

Given $(M,g) \in \mathcal{M}(d,\kappa)$, $\Lambda > 0$ and a function $u \in
\Ran(\Pi_\Lambda)$, we define, via the functional calculus, a harmonic extension of $u$ by
\begin{equation}
	\label{eq: harmonic extension}
    w(t,x) = \frac{\sinh(t \sqrt{- \Delta_g})}{\sqrt{-\Delta_g}} u,
\end{equation}
which satisfies
\begin{equation}
	(\partial_t^2 + \Delta_g) w = 0 \text{ in } \mathbb{R} \times M,
\end{equation}
as well as boundary conditions suitable to applying elliptic propagation of smallness estimates to $w$. A similar harmonic extension of $e^{s\Delta_g}u$ is used in the proof of Theorem \ref{theo: fixed time obs}.

We then apply Logunov and Malinnikova's propagation of smallness inequalities
from co-dimension 1 subsets for gradients of harmonic functions \cite[Theorem
5.1]{logunovmalinnikova2018qtttvepropagsmallness} 
in carefully constructed charts. However, to obtain a uniform local inequality, these charts need to satisfy the following conditions:
\begin{enumerate}[(i)]
	\item \label{item: LM domain} the domain of the charts must be the same,
	\item \label{item: LM metric} the metric must have uniform ellipticity and
        Lipschitz constants in those charts,
	\item \label{item: LM measure} the measure of the sets from which smallness is propagated must be uniformly bounded from below.
\end{enumerate}

We solve issues \eqref{item: LM domain} and \eqref{item: LM metric} using 
harmonic coordinates. For the metric to be sufficiently regular in those
coordinates we would require at least a bound on the Ricci
curvature. However, the construction of harmonic coordinates
also requires a lower bound on the local injectivity radius. 
Besides, point \eqref{item: LM measure} cannot be straightforwardly implied by
\eqref{eq: thickness} due to the possibly small denominator $\Vol(B(x,R))$
(for example, $\Vol(B(x,R))$ vanishes asymptotically when $x$ tends towards
the end of a cusp in a hyperbolic surface). We solve both issues at once by
constructing good charts in a local covering space of the manifold by the
exponential map, the \textit{normal cover}, which requires us to strengthen our
assumption to bounded sectional curvature.

The norm bound on the sectional curvature enters play in two different ways.
First, as a lower bound on sectional curvature which provides a key volume doubling inequality
(Proposition \ref{prop:vol_db}) that is repeatedly used throughout the article,
notably for counting arguments. Then, as an upper bound on the sectional
curvature which plays a role in the construction
of the aforementioned harmonic coordinates.

We note that replacing the Logunov--Malinnikova inequality with \cite[Theorem
2.2]{lebalchmartin2024quantitative} generalises our result to
operators $H = -\Delta_g + V$ with potentials $V \in \RL^\infty(M)$ with the
caveat that there is an
additional dependence of $C_{\rm spec}$ on $\|V\|_{\RL^\infty}$.

The article is divided as follows: in Section \ref{sec:good-coord-charts}, we
recall geometric preliminaries, and define and investigate the $\mathcal{M}(d,\kappa)$ class of
manifolds and the uniform properties of their normal covers.
In Section \ref{sec:prop-smalln-lipsch}, we prove the local
elliptic estimates based on the Logunov--Malinnikova inequality and extend them
uniformly within the class $\CM(d,\kappa)$. Section \ref{sec:endgame} contains the conclusion
of the proof of Theorems \ref{theo: main result} and \ref{theo: fixed time
obs}.

\subsection*{Acknowledgments}

A.D. and M.R. acknowledge funding from the ERC-2022-ADG grant GEOEDP (project
101097171). J.L. would like to thank the Isaac Newton Institute for
Mathematical Sciences, Cambridge, for support and hospitality during the
programme Geometric spectral theory and applications, where the final stages of this paper
were written. This work was supported by EPSRC grant EP/Z000580/1.

The authors would like to thank Nicolas Burq, Paul Alphonse, and
Eugenia Malinnikova for interesting discussions concerning the topic of this article.


\section{Geometric preliminaries}
\label{sec:good-coord-charts}

\subsection{Geometric quantities and notation}
Let $M$ be a $d$-dimensional smooth manifold. A Riemannian metric $g$ on $M$ is
an assignment, for every $x \in M$, of a positive definite bilinear form on the
tangent space $T_xM$, that is a positive section of $T^*M \otimes T^*M$.
Locally, the metric $g$ has the coordinate expression
\begin{equation}
    g(x) = \sum_{j,k} g_{jk}(x) \dd x^j \otimes \dd x^k,
\end{equation}
where $\{\dd x^j : 1 \le j \le d\}$ is a smoothly varying basis for the
cotangent space. We denote by $g^{jk}(x)$ the local coordinate expression for
the inverse of the metric and regularly omit the dependence of the metric (and
of other tensors) on the base point $x$. While the metric is an intrinsic
object, in the sense that the inner product on $TM$ is independent of the choice
of local coordinates, its regularity can only be inspected in coordinates and is
defined in a chart as the regularity of the coefficients $g_{jk}$. The main
control we need is given as follows, of course defined for the expression of a
metric in a chart.
\begin{definition}
    Let $g$ be a Riemannian metric defined on some open set $\Omega \subset \R^d$. The
    \emph{biLipschitz constant} $\blip(g)$ is defined as
    \begin{equation}
        \blip(g) :=
        \max\set{\|g\|_{\RW^{1,\infty}(\Omega)},\|g^{-1}\|_{\RL^\infty}(\Omega)}.
    \end{equation}
    If $g$ is defined on a manifold, its biLipschitz constant is the infimum
    over all coverings of $M$  by charts of the supremum of the biLipschitz
    constants read in each of those charts.
\end{definition}
For $x \in M$
and $r > 0$, we denote the ball of radius $r$ around $x$ by $B_x(r)$.
Since we will repeatedly use alternate between balls in $M$ and balls in $T_xM$
we denote the ball of radius $r$ around $0\in T_xM$ by $\tilde B_x(r)$. In
particular, this data gives us one example of a geometrically defined chart around a point. 
    \begin{definition}
        Given a Riemannian manifold with Lipschitz metric $(M,g)$ and $x \in M$, the
        \emph{exponential map} $\exp_x : T_xM  \to M$ assigns to $v \in T_xM$
        the position $\exp_x(v)$ at time $1$ of the geodesic with initial data
        $(x,v)$ at time $0$. The \emph{injectivity radius} at $x$ is the maximal
        $r > 0$ such that the restriction $\exp_x : \tilde B_x(r) \to B_x(r)$ is
        injective.
    \end{definition}

The Laplacian on $M$, also known as the Laplace--Beltrami operator,
is given in local coordinates as
\begin{equation}
    \Delta_g u = \frac{1}{\sqrt{|\det g|}} \sum_{j,k} \del_j \left(\sqrt{|\det
    g|} g^{jk} \del_k\right),
\end{equation}
in other words it is represented locally as an elliptic operator in divergence
form, and any theorem stated in Euclidean space for such an operator can be
applied wholesale in charts to the Laplace--Beltrami operator.

Two important intrinsic geometric quantities associated with the metric are the
\emph{Ricci curvature} $\operatorname{Ric}$ and the \emph{sectional curvature}
$K$. Since we will not be
working directly with the expression of those curvatures, but rather with
quantities explicitly controlled by bounds on the curvature, we neither provide a
precise definition nor describe the difference between both, other than
observing that bounds on the sectional curvature imply bounds on the Ricci
curvature. We use bounds on the sectional curvature to define classes of manifolds
without referring to local charts, we will point out which of our lemmas
only require bounds on the Ricci curvature.

  \begin{definition}\label{def:bdcurv}
    Let $\kappa\in [0,+\infty)$ and $d\geq 1$. We denote by 
    \begin{equation}
        \label{eq: sectional curvature bound}
        \CM(d,\kappa) := \set{(M,g) : \operatorname{dim}(M) = d \text{ and }
        \|K\|_{\RL^\infty} \le \kappa}
    \end{equation}
     the set of $d$-dimensional
     manifolds with (sectional) curvature bounded by $\kappa$. Note that
     elements of $\CM(d,\kappa)$ are neither assumed to be compact nor
     connected.
  \end{definition}

  \subsection{Volume and density bounds}
Manifolds of bounded curvature satisfy a volume doubling inequality, which in
turn implies bounds on the density of uniformly discrete sets.
All results in this subsection depend, in fact, only on lower bounds for the
Ricci curvatures. The doubling inequality can  be read directly off the
  Bishop--Gromov inequality \cite[Lemma 5.3.bis]{gromov2007metricstructures},
  and we state it for convenience of the reader.
  \begin{proposition}\label{prop:vol_db}
      Let $\kappa>0$ and $d \in \mathbb{N}$. There is $C > 0$ such that for every $(M,g) \in
      \CM(d,\kappa)$, $x \in M$ and $R > r > 0$, we have the inequality
    \begin{equation}
    	\label{eq: weak volume doubling}
        {\rm Vol}(B_x(R))\leq C \left(\frac R r\right)^d e^{(d-1) \sqrt \kappa(R - r)} {\rm Vol}(B_x(r)).
    \end{equation}
  \end{proposition}
  Proposition \ref{prop:vol_db} allows us to control the
  density of discrete sets in $M$.
  \begin{proposition}\label{prop:covering_balls}
      Let $d \in \N$, $\kappa > 0$. There is $C >
      0$ such that for every $r, R > 0$, any $M \in \CM(d,\kappa)$, and every $r$-separated set $\BY
      \subset M$,
      \begin{equation}
          \#(\BY \cap B_x(R)) \le c \left(\frac R r \right)^d e^{2(d-1) \sqrt \kappa
          R}
      \end{equation}
      for all $x \in M$.
  	\end{proposition}
  	\begin{proof}
        Let $r > 0$ and $\BY \subset M$ be an $r$-separated set. For all $y,y' \in \BY$,
        \begin{equation}
            y \ne y' \quad \Longrightarrow \quad B_y(r/2) \cap B_{y'}(r/2) =
            \varnothing,
        \end{equation}
        so that for any $x \in M$ and $R > 0$
        \begin{equation}
            \label{eq:subset}
            \begin{aligned}
                \sum_{y \in \BY \cap B_x(R)}  \Vol(B_y(r/2)) 
            \le  \Vol(B_x(R + r/2))  
            \le C e^{(d-1) \sqrt \kappa \tfrac r 2} \Vol(B_x(R)).
        \end{aligned}
        \end{equation}
        On the other hand, $B_x(R) \subset B_y(2R)$ for every $y \in \BY \cap
        B_x(R)$.
        Thus,
        \begin{equation}
            \label{eq:doublingcovering}
            \Vol(B_x(R)) \le \min_{y \in \BY \cap B_x(R)} \Vol(B_y(2R)) \le C
            \left(\frac{R}{r}\right)^d e^{(d-1)
            \sqrt{\kappa} (2R - \tfrac r 2)}  \min_{y \in
            \BY \cap B_x(R)}  \Vol(B_y(r/2)).
        \end{equation}
        Combining \eqref{eq:subset} and \eqref{eq:doublingcovering}, we get
        that, up to maybe choosing a larger $C$,
        \begin{equation}
            \#(\BY \cap B_x(R)) \le C\left(\frac{R}{r}\right)^d e^{2(d-1) \sqrt{\kappa} R}.
            \end{equation}
  	\end{proof}
    We note that on spaces that satisfy the doubling inequality \eqref{eq: weak volume
    doubling} the thickness property of Definition \ref{def:thickness} is
    well-behaved, in the sense that if a set is thick at scale $R_0$ it is also
    thick at any scale $R > 0$. Furthermore, if $M$ has no compact connected
    component, thickness of a subset $\omega \subset M$ is invariant by removal
    of any compact set in $M$.
  \subsection{Normal covers and good coordinate systems}
  \label{sec:coverings}

We now aim at constructing local coordinate charts that are uniformly controlled
within the classes $\CM(d,\kappa)$. In particular, the aim is that
  \begin{enumerate}
  	\item the metric is controlled only by $d$ and
        $\kappa$ in those charts,
    \item the size of the neighbourhoods on which the charts map depends only
        $\kappa$.
    \end{enumerate} 
    The injectivity radius is conspicuous by its absence from this list of
    requirements. Indeed, if we allow dependence on the injectivity radius as
    well there are many standard choices of coordinate charts which could
    satisfy our requirements depending on the level of regularity required, for
    instance harmonic coordinates or the ``almost linear coordinates'' developed
    by Jürgen Jost \cite{jostkarcher1982ingerman,jost1984harmonicmappingsbook}.
    These coordinates can only be defined on a neighbourhood whose size is
    bounded by the injectivity radius. In the class of examples we consider, the
    injectivity radius is of course everywhere positive, however it may not be
    bounded away from zero, for instance in the case of a hyperbolic manifold
    with cusps.

    In order to address this difficulty, we will instead use harmonic coordinates
    on a local cover of $M$ around $x$; the \emph{normal cover}, and
    it is on this cover that we will perform our analyses. The
    injectivity radius of the normal cover depends only on $\kappa$,
    so that the harmonic coordinates will be defined on balls whose radius is
    uniformly controlled by $\kappa$.  
     The price to pay is that the number of pre-images of a point in
    the normal cover may grow as the injectivity radius goes to zero, but it
    does so in a quantifiable way which lets us control estimates appropriately.

The exponential map provides a cover of the connected component of $M$
containing $x$ by $T_xM$, whose metric is given by the pullback $\exp_x^* g$.
However, this metric is usually not well-behaved on the whole of $T_xM$, and it is
for this reason that we will define it on a smaller ball. The first step is to
show that there is a ball in $T_xM$ in which the injectivity radius of $\exp_x^*
g$ is uniformly controlled.
\begin{proposition}
    \label{prop:rauch}
    Let $\kappa > 0$. There is $L > 0$ such that for every $d \in \N$, $M \in
    \CM(d,\kappa)$ and $x \in M$, the restriction
    \begin{equation}
        \exp_x : B_{T_xM}(0, \tfrac{\pi}{2 \sqrt \kappa}) \to M
    \end{equation}
    is locally bi-$L$-Lipschitz. Furthermore, in the metric $\exp_x^* g$, every $v \in
    B_{T_xM}(0,\tfrac{\pi}{4
    \sqrt \kappa})$ has injectivity radius larger than $\tfrac{\pi}{4
    \sqrt{\kappa}}$.
\end{proposition}

\begin{proof}
BiLipschitzness follows from Rauch's comparison theorem
\cite{rauch1951contributiondifferentialgeometry}, (a modern exposition is found
    in 
\cite[Lemma 2.2.1]{jost1984harmonicmappingsbook}). By \cite[Lemma
2.4.2]{jost1984harmonicmappingsbook}, $B_{T_xM}(0,\tfrac{\pi}{2 \sqrt \kappa})$ is geodesically
convex under the metric pulled back by the exponential map, meaning that any two
points are joined by a unique geodesic segment, directly
implying that the injectivity radius is bounded below by $\tfrac{\pi}{4\sqrt
\kappa}$ on the ball of radius $\tfrac{\pi}{4\sqrt \kappa}$.
\end{proof}
While the exponential map is itself biLipschitz, it does not provide us
with uniform control over the biLipschitz constant of $g$. As first noticed by
DeTurck--Kazdan \cite{deturckkazdan1981regularity}, the right choice of
coordinates for optimal control of the regularity are \emph{harmonic
coordinates}. The following proposition collates statements found in the
discussion before Lemma 2.8.1, and in Lemma 2.8.2 of
\cite{jost1984harmonicmappingsbook} applied to $T_xM$.
  \begin{proposition}\label{prop:good_charts}
      Let $\kappa>0$ and $d \in \mathbb{N}$. There is $L, \rho > 0$ and $r_\kappa > 0$, called the
      \emph{uniform Lipschitz harmonicity radius} so that for every $M \in \CM(d,\kappa)$
      and $x \in M$, there is a harmonic $\phi : \tilde B_x(r_\kappa) \to
      \tilde B_{x}(\tfrac{\pi}{4 \sqrt \kappa})$, diffeomorphic on its image so
      that $\tilde g := \phi^* \exp_x^* g$ is biLipschitz with $\blip(\tilde g) \le
      L$. Moreover, $\tilde g$ is the canonical metric at $0$, and the image of
      $\phi$ contains $B_{T_xM}(0,\rho)$.
  \end{proposition}
  Of course, the harmonic radius usually depends on the injectivity radius of
  the manifold, but as shown in Proposition \ref{prop:rauch}, the injectivity
  radius on the tangent space can be bounded below in terms of the curvature.
  The normal cover at a point is simply a ball in the tangent space
  whose radius is slightly smaller than the harmonic radius.
  \begin{definition}\label{def:normal_cover}
      Let $\kappa>0$, $d\in \mathbb{N}$, $M\in
      \mathcal{M}(d,\kappa)$, and $x\in M$. The \emph{normal cover}
    of $M$ at $x$ is $\tilde M_x := \tilde B_x(r_\kappa)$, endowed
    with the pullback metric $\tilde g := \phi^* \exp_x^* g$ expressed in
    harmonic coordinares, and natural projection
    $p_x : \tilde M_x \to M$. 
    For every
    quantity $Q$ defined in a neighbourhood of $x$ (e.g. a subset, a function,
    the metric, etc.) we denote $\tilde Q$ to be the lift to $\tilde M_x$ of
    $Q$,
  \end{definition}

  While the injectivity radius of $\tilde M_x$ is bounded away from zero
  independently of $x$, the
  degree of the covering $\tilde M_x$ grows unboundedly as $x$ varies in a region where the
  injectivity radius of $M$ goes to zero. It turns out that as long as the
  projection of balls centred at $0$ in $\tilde M_x$ is distributed uniformly
  enough on the base space we can recover enough control. In order to make this quantitative, we introduce the
  following two quantities.
  \begin{definition}
      \label{def:counting}
      For any $x \in M$ and $r \in (0,\tfrac{r_\kappa}{4})$, we define
      the upper pre-image counting function
    \begin{equation}
    	\label{eq: upper preimage counting function}
        N(x,r) := \max_{y \in B_x(r)} \# (p_x^{-1}(y) \cap \tilde B_{x}(r))
    \end{equation}
    and the lower pre-image counting function
    \begin{equation}
    	\label{eq: lower preimage counting function}        n(x,r) := \min_{y \in B_x(r)} \# (p_x^{-1}(y) \cap \tilde B_{x}(4r))
    \end{equation}
  \end{definition}

\begin{proposition}\label{prop:counting}
For every $\kappa>0$ and $d\in \mathbb{N}$ there is $C > 0$ so that
the pre-image counting functions are equivalent in the sense that.
\begin{equation}
    N(x,r) \le n(x,r) \le Ce^{3(d-1)\sqrt \kappa r} N(x,r),
\end{equation}
Furthermore, the pre-image counting functions satisfy the doubling estimates
\begin{equation}
    N(x,2r) \le C e^{7(d-1) \sqrt \kappa r} N(x,r) \qquad \text{ and } \qquad
    n(x,2r)
    \le C e^{7(d-1) \sqrt \kappa r} N(x,r).
\end{equation}
\end{proposition}

  \begin{proof}
      Consider $y,z \in B_x(r)$ such that $y$ is maximizing in \eqref{eq: upper preimage counting function} and $z$ is minimizing in \eqref{eq: lower preimage counting function}. Let $\gamma$ be a path
    of length at most $2r$ connecting $y$ and $z$ and staying in
    $B_x(r)$. The lifts $\tilde \gamma$ of $\gamma$ have the same length and connect
    injectively preimages $\tilde y \in p_x^{-1}(y)$ and $\tilde z \in p_x^{-1}(z)$.
    Assuming that $\tilde y \in \tilde B_x(r)$, the triangle inequality gives us
    $\tilde z \in \tilde B_x(4r)$, so that indeed $N(x,r) \le n(x,r)$.
    Now, we observe that taking $y = x$ in \eqref{eq: upper preimage counting function} and \eqref{eq: lower preimage counting function} yields 
    \begin{equation}
        \label{eq:volumecomp}
        n(x,r) \le 
        \frac{\Vol(\tilde B_x(4r))}{\Vol(B_x(r))}
        \quad \text{and}
        \quad N(x,r) \ge 
        \frac{\Vol(\tilde B_x(r))}{\Vol(B_x(r))}.
    \end{equation}
    Combining \eqref{eq:volumecomp} and Proposition \ref{prop:vol_db}, there is
    $C > 0$ so that
    \begin{equation}
        n(x,r) \le \frac{\Vol(\tilde B_x(4r))}{\Vol(\tilde B_x(r))} N(x,r) \le C
        e^{3 (d-1) \sqrt \kappa r} N(x,r).
    \end{equation}
    Similarly, we obtain
    \begin{equation}
        N(x,2r) \le   n(x,2r) \le \frac{\Vol(\tilde B_x(8r))}{\Vol(\tilde B_x(r))} N(x,r) \le C
        e^{7 (d-1) \sqrt \kappa r} N(x,r),
    \end{equation}
    providing the doubling estimates claimed.
  \end{proof}

  \begin{remark}
    The sectional curvature bound was only really used in the Rauch comparison
    theorem to obtain a lower bound on the injectivity radius of the normal
    cover. As indicated in \cite[Theorem
    6]{hebeyherzlich1997harmoniccoordinates}, the uniform Lipschitz harmonic radius can be bounded below in terms
    of the $\RL^\infty$ norm of the Ricci curvature and the injectivity radius,
    so that one can replace sectional curvature with Ricci curvature if there is
    an a priori lower bound on the injectivity radius of the normal cover. Note,
    however, that by the same Theorem, we cannot extend this to lower bounds on the Ricci curvature,
    as that would only guarantee bounds on the $\RC^{0,\alpha}$ norm of metric in harmonic
    coordinates, which is not sufficient to later on apply the
    Logunov--Malinnikova propagation of smallness estimate.
  \end{remark}

  We conclude this section with a few remarks concerning
    \emph{a priori} regularity of the metric.
    
    The works
    \cite{rauch1951contributiondifferentialgeometry,jostkarcher1982ingerman,jost1984harmonicmappingsbook},
    and most of the literature built thereupon, assume the Riemannian
    metric $g$ to be smooth in the original chart.

    By density, all arguments laid down in this section extend
    directly to the case of $\RC^2$ metrics with bounded sectional
    curvature: at least locally, these metrics are $\RC^2$ close to
    smooth metrics with a (slightly worse) bound on the sectional
    curvature. Then, geodesic rays have a Lipschitz dependence on $\RC^2$
    variations on the metric, and therefore Proposition
    \ref{prop:rauch} extends to 
    this situation; in conclusion, our construction of normal
    coverings applies to this case. Then, one can use again density
    arguments: after a normal covering, there exists a closeby smooth
    metric with bounded curvature, and in harmonic coordinates for the
    smoothed metric are good enough, the initial metric will be
    controlled in the Lipschitz topology.

    More generally, let $g_1$ be a smooth metric with bounded sectional
    curvature, and let $g_2$ be a Lipschitz metric which is \emph{comparable}
    in the sense that $cg_1\leq g_2\leq Cg_1$ everywhere for some
    $0<c<C$, and which satisfies
    $\|\nabla_{g_1}g_2\|_{\RL^{\infty}}<\infty$, a family of charts on
    coverings for which $g_1$ is controlled in the Lipschitz topology
    will also control $g_2$ in the Lipschitz topology. In conclusion,
    our main results will also apply to $g_2$ (notice that the
    thickness condition is preserved by changing $g_1$ for $g_2$).

    Unfortunately,
    there does not seem to be an easy way to characterise Lipschitz
    metrics satisfying this condition for some $g_1$, as there is no
    notion of sectional curvature at this regularity. A necessary
    condition is the volume doubling estimate, generalising
    Proposition \ref{prop:vol_db}.

	\section{Elliptic tools}
\label{sec:prop-smalln-lipsch}

This section collects the essential tools towards the proof of Theorem
\ref{theo: main result}. First, we recall some classical interior estimates for harmonic functions. Then, we obtain a
propagation of smallness result from the famous Logunov--Malinnikova theorem.

\begin{lemma}
    \label{lem:harmoniclipschitz}
    Let $U \Subset \Omega \Subset \R^d$ and $L > 0$. There
    is $C > 0$ so that for every metric $g$ on $\Omega$ with $\Lip(g) < L$ and any $\Delta_g$-harmonic functions $w : \Omega \to \R$
    we have the bound
    \begin{equation}
        \norm{w}_{\RC^1(U)} \le C\norm{w}_{\RL^2(\Omega)}. 
    \end{equation}
\end{lemma}

\begin{proof}
    Let $U_1$ be such that $U \Subset U_1 \Subset \Omega$. It follows from \cite[Theorem 9.11]{gilbargtrudinger2001bookelliptic} that 
     for all $p \in
    (1,\infty)$,
    \begin{equation}
        \label{eq:gilbargtrudingerelliptic}
        \|w\|_{\RW^{2,p}(U_1)}
        \lesssim_{U,\Omega,p,d,\Lip(g)}
        \|w\|_{\RL^p(\Omega)},
    \end{equation}
    and, applying the Sobolev embedding theorem with $p = 2$ to $U_1$ to see that
    \begin{equation}
        \label{eq:gtaftersob}
        \|w\|_{\RL^{\tfrac{2d}{d-4}}(U_1)}
        \lesssim_{U,\Omega,d,\Lip(g)}
         \|w\|_{\RL^2(\Omega)}.
    \end{equation}
    Let $k = \lceil \tfrac{d-2}{4}\rceil$, and consider a sequence of sets such
    that $U \Subset U_k
    \Subset \dotso \Subset U_1 \Subset \Omega$. Applying recursively
    inequalities \eqref{eq:gilbargtrudingerelliptic} and \eqref{eq:gtaftersob}
    to pairs $U_j, U_{j-1}$ we have that for some $q > d$,
    \begin{equation}
        \|w\|_{\RL^q(U_k)} 
        \lesssim_{U,\Omega,d,\Lip(g)}
        \|w\|_{\RL^2(\Omega)}.
    \end{equation}
    Applying \eqref{eq:gilbargtrudingerelliptic} one last time between $U$ and
    $U_k$, but this time
    followed by the Morrey embedding theorem, then reveals
    \begin{equation}
        \|w\|_{\RC^1(U)} \lesssim_{U, \Omega, d, \Lip(g)} \|w\|_{\RL^2(\Omega)}.
    \end{equation}
\end{proof}

We now deduce appropriate propagation of smallness estimates from those of
Logunov--Malinnikova \cite{logunovmalinnikova2018qtttvepropagsmallness}. We
start by deducing a local $\RL^2$ estimate.
\begin{proposition}\label{prop:LML2}
    Let $d\geq 2$ and let $\Omega \Subset U \Subset \R^d$, and $T,\epsilon, L>0$. There exists $C > 0$ and, 
    $\alpha \in (0,1)$
    such that every locally Lipschitz metrics $g$ on $U$ with $\Lip(g) < L$,
    and all pairs $u : \Omega
    \to \R$ and $w:U
    \times [-T,T]\to \mathbb{R}$
	solving
    \begin{equation}
        \begin{cases}
        (\Delta_g + \partial_t^2)w = 0 & \text{in } U \times [-T,T] \\
        w(x,0) \equiv 0  & \text{and } \del_t w(x,0) = u(x) 
        \end{cases}
    \end{equation}
satisfy
	\[
	\|u\|_{\RL^2(\Omega)}\leq
    C\|u\|_{\RL^2(\omega)}^{\alpha}\|w\|_{\RL^2(U \times [-T,T])}^{1-\alpha}
	\]
    for every $\omega \subset \Omega$ such that $\CH^d(\omega) > \eps$.
\end{proposition}
\begin{proof}
    In this proof, we abuse notation by using the same letter for $U \subset
    \R^d $ and $U \times \{0\} \subset \R^d_x \times \R_t$. On $\R^d_x \times
    \R_t$ we put the metric $g' = g_x + \rd t^2$. We note that since $w$ is
    constant on $U$, we have that $|u(x)| = |\nabla w(x,0)|$ for all $x \in U$
    and as such it is that quantity that we estimate. In particular, we assume
    that $\nabla w \not \equiv 0$ on $\Omega$, as otherwise the result holds
    trivially. Let $V$ be such that
    $\Omega \Subset V \Subset U$ and let $\omega \subset
    \Omega$ be such that $\CH^d(\omega) > \eps$. For $a > 0$ put
   $E_a =  \{x \in \omega : |\nabla w(x,0)| \ge a \}$
    so that by Chebyshev's inequality, 
    \begin{equation}
        \label{eq:chebyshev}
        \CH^{d}(E_a)^{1/2} \le
        \frac 1 a\norm{\nabla w}_{\RL^2(\omega)}.
    \end{equation}
    If $\CH^{d}(E_a) < \tfrac \eps 2$, the Logunov--Malinnikova
    propagation estimate \cite[Theorem
    5.1]{logunovmalinnikova2018qtttvepropagsmallness} applied to $\omega
    \setminus E_a \Subset \Omega \Subset V$
    (which then has measure
    larger than $\tfrac \eps 2$) leads us to
    \begin{equation}
        \|\nabla w\|_{\RL^\infty(\Omega)} \le C \|\nabla w\|_{\RL^\infty(\omega
        \setminus E_a)}^\alpha \|\nabla w\|_{\RL^\infty(V)}^{1-\alpha} \le C a^
        \alpha \|\nabla w\|_{\RL^\infty(V \times [-T,T])}^{1-\alpha},
    \end{equation}
    or, rearranging,
    \begin{equation}
        \label{eq:aifEasmall}
        a^\alpha \ge \frac{\|\nabla w\|_{\RL^\infty(\Omega)}}{C\|\nabla w\|^{1 -
        \alpha}_{\RL^\infty(V \times [-T,T])}}.
    \end{equation}
    Contrapositively, if $a$ doesn't satisfy inequality \eqref{eq:aifEasmall}
    then $\CH^{d}(E_a) > \tfrac \eps 2$. Inserting this assumption in \eqref{eq:chebyshev} and
    taking powers of $\alpha$ on both sides gives us
    \begin{equation}
        \|\nabla w\|_{\RL^\infty(\Omega)} \le C
        \left(\frac{2}{\eps}\right)^{\frac \alpha 2} \|\nabla
        w\|_{\RL^2(\omega)}^\alpha \|\nabla w\|_{\RL^\infty(V \times
    [-T,T])}^{1-\alpha}.
    \end{equation}
    By definition, then following from Lemma \ref{lem:harmoniclipschitz}, we see
    that
    \begin{equation}
        \|\nabla w\|_{\RL^\infty(V \times [-T,T])} \le \|w\|_{\RC^1(V \times
        [-T,T])} \lesssim \|w\|_{\RL^2(U \times [-T,T])},
    \end{equation}
    whence the simple application of H\"older's inequality
    \begin{equation}
        \|u\|_{\RL^2(\Omega)} = \|\nabla w\|_{\RL^2(\Omega)} \le
        \Vol(\Omega)^{\tfrac 1 2} \|\nabla w\|_{\RL^\infty(\Omega)}
    \end{equation}
        completes the proof.
\end{proof}

In the remainder of this section, we extend Proposition \ref{prop:LML2} to
manifolds of bounded curvature in $\CM(d,\kappa)$. We proceed in
three steps, starting with a local estimate in small enough balls then extending
to balls of any fixed radius via a covering argument. Finally, we extend the statement to any
$M \in \CM(d,\kappa)$ under the assumption that $\omega$ is $(R,\eps)$-thick
inside $M$. Recall the
conventions from Definition \ref{def:normal_cover}. Namely, given $(M,g) \in
\CM(d,\kappa)$, the normal cover at $x \in M$ is denoted $\tilde M_x$, we write
$B_x(r) = B_M(x,r)$, and for any quantity $s$ associated to $M$ (be it a function,
a subset, the metric, etc.), we denote $\tilde s$ the associated quantity on
$\tilde M_x$.
\begin{proposition}\label{prop:LM_smallballs}
    Let $d\geq 2$, $\kappa, \eps, T >0$, and $r \in (0, \tfrac{r_\kappa}{4})$. There
    exist $C > 0$ and $\alpha \in (0,1)$ so that every $(M,g) \in
    \CM(d,\kappa)$, $x \in M$, and pairs $u : B_x(\tfrac{r_\kappa}{4}) \to
    \R$ and $w : B_x(\tfrac{r_\kappa}{4}) \times [-T,T]$ solving
    \begin{equation}
        \begin{cases}
            (\Delta_g + \partial_t^2)w = 0 & \text{in } B_x(\tfrac{r_\kappa}{4}) \times [-T,T] \\
        w(y,0) \equiv 0  & \text{ and } \del_t w(y,0) = u(y) 
        \end{cases}
    \end{equation}
    satisfy the estimate
	\[
        \|u\|_{\RL^2(B_x(r))}\leq
        C\|u\|_{\RL^2(\omega)}^{\alpha}\|w\|_{\RL^2(B_x(\tfrac{r_\kappa}{4}) \times [-T,T])}^{1-\alpha}
	\]
    for every $\omega \subset B_x(r)$ such that $\Vol(\omega) > \eps
    \Vol(B_x(r))$.
\end{proposition}
\begin{proof}
    Let $\omega
    \subset B_x(r)$ be such that $\Vol(\omega) > \eps \Vol(B_x(r))$.  By
    Proposition \ref{prop:counting},
    \begin{equation}
        \begin{aligned}
            \Vol(\tilde \omega \cap \tilde B_x(4r)) &\ge \Vol(\omega) n(x,r) \\
            &\ge \eps \Vol(B_x(r)) N(x,r)\\
            &\ge \eps \Vol(\tilde B_x(r)).
    \end{aligned}
    \end{equation}
    Proposition \ref{prop:good_charts} tells us that $\Vol(\tilde B_x(r))$ is bounded
    away from zero and that $\Lip(\tilde g)$ is bounded, depending only on $r,
    \kappa$ and $d$. Applying Proposition \ref{prop:LML2}
    with $\Omega = \tilde B_x(4r) \Subset U = \tilde B_x(\tfrac{r_\kappa}{4})$ we obtain $C > 0$ and
    $\alpha \in (0,1)$ so that
	\[
        \|\tilde{u}\|_{\RL^2(\tilde B_x(4r))}\leq C
\|\tilde{u}\|_{\RL^2(\tilde{\omega})}^{\alpha}\|\tilde{w}\|_{\RL^2(\tilde
B_x(\tfrac{r_\kappa}{4})
\times [-T,T])}^{1-\alpha}.\]
Now, a simple counting argument combined with H\"older's inequality gives on the one hand
	\[
        n(x,r)^{\frac 12}\|u\|_{\RL^2(B_x(r))}\leq
        \|\widetilde{u}\|_{\RL^2(\tilde B_x(4r))}
\]
whilst on the other hand
\[
    \|\widetilde{u}\|_{\RL^2(\widetilde{\omega})}\leq N(x,\tfrac{r_\kappa}{4})^{\frac{1}{2}}\|u\|_{\RL^2(\omega)}\quad \text{and}
\quad
\|\widetilde{w}\|_{\RL^2(\tilde B_x(\tfrac{r_\kappa}{4}))}\leq
N(x,\tfrac{r_\kappa}{4})^{\frac 12}\|w\|_{\RL^2(B_x(\tfrac{r_\kappa}{4})
\times [-T,T])},\]
	and therefore
	\[
        \|u\|_{\RL^2(B_x(r))}\le C\left(
            \frac{N(x,\tfrac{r_\kappa}{4})}{n(x,r)}
    \right)^{\frac 12}
    \|u\|_{\RL^2(\omega)}^{\alpha}\|w\|_{\RL^2(B_x(\tfrac{r_\kappa}{4})\times [-T,T])}^{1-\alpha}.
	\]
	Applying repeatedly Proposition \ref{prop:counting}, we find that the ratio $
    \frac{N(x,\tfrac{r_\kappa}{4})}{n(x,r)}$ is bounded depending only on $r, d, \kappa$, which concludes the proof.
\end{proof}

We now propagate this inequality from small to large balls.

\begin{proposition}
	\label{prop: LM_largeballs}
    Let $d \geq 2$ and $\kappa, R, \eps, T > 0$ and put $R' = R +
    r_\kappa$. There exist $C > 0$ and $\beta
    \in (0,1)$ so that for every $(M,g)\in \mathcal{M}(d,\kappa)$, $x\in M$ and
    pairs $u: B_x(R) \to \R$ and $w : B_x(R) \times [-T,T] \to \R$ solving
    \begin{equation}
        \begin{cases}
            (\Delta_g + \partial_t^2)w = 0 & \text{in } B_x(\tfrac{\pi}{8 \sqrt
            \kappa}) \times [-T,T] \\
        w(y,0) \equiv 0  & \text{ and } \del_t w(y,0) = u(y) 
        \end{cases}
    \end{equation}
    we have the estimate
	\[
        \|u\|_{\RL^2(B_x(R))}\leq
C\|u\|_{\RL^2(\omega)}^{\beta}\|w\|_{\RL^2(B_x(R') \times [-T,T])}^{1-\beta}
	\]
    for every $\omega \subset B_x(R)$ such that $\Vol(\omega) > \eps
    \Vol(B_x(R))$.
\end{proposition}

\begin{proof}
    Let $r = \tfrac{r_\kappa}{16}$ and $\BY$ be a maximal $r/2$-separated
    set in $B_x(R)$, in particular $B_x(R)$ is covered by balls centred on
    elements of $\BY$ of radius $r/2$. Let $\BG$ be the graph whose set of vertices is $\BY$, and
    where the edges are given by the relation
    \begin{equation}
        y \sim y' \Longleftrightarrow B_y(r) \cap B_{y'}(r) \text{ contains a
        ball of radius } r/2.
    \end{equation}
    Since $\BY$ is a maximal $r/2$-separated set in the connected set $B_x(R)$, $\BG$ is a connected graph,
    and by Proposition \ref{prop:covering_balls}, the cardinality of its vertex
    set is bounded above by a constant $N$ depending only on $r,\kappa,R$ and
    $d$. 

There is $y_0 \in
    \BY$ such that
    \begin{equation}
        \Vol(\omega \cap B_{y_0}(r/2)) \ge \frac \eps N \Vol(B_x(R)) \ge \frac \eps
        N \Vol(B_{y_0}(r)).
    \end{equation}
    In particular, it follows from Proposition \ref{prop:LM_smallballs} that
    there are $C > 0$ and $\alpha \in (0,1)$ such that
    \begin{equation}
        \|u\|_{\RL^2(B_{y_0}(r))} \le C \|u\|_{\RL^2(\omega)}^\alpha \|w\|^{1 -
            \alpha}_{\RL^2(B_x(R') \times [-T,T])}.
    \end{equation}
    Let $y \in \BY$ be such that $y \sim y'$. Applying Proposition
    \ref{prop:LM_smallballs}, with the smaller set $\omega'$ being the ball of
    radius $r/4$ with same center as the ball of radius $r/2$ in the intersection, we get
    \begin{equation}
        \begin{aligned}
            \|u\|_{\RL^2(B_{y}(r))} &\le 
            C \|u\|_{\RL^2(\omega')}^\alpha \|w\|^{1 -
            \alpha}_{\RL^2(B_x(R') \times [-T,T])}
            \\
            &\le 
            C \|u\|_{\RL^2(\omega)}^{\alpha^2} \|w\|^{1 -
            \alpha^2}_{\RL^2(B_x(R') \times [-T,T])}.
        \end{aligned}
    \end{equation}
    Applying Proposition \ref{prop:LM_smallballs} repeatedly from neighbours to
    neighbours in $\BG$ in the same way, we obtain in the end that for any $y \in \BY$,
    \begin{equation}
        \label{eq:ineveryball}
        \|u\|_{\RL^2(B_y(r))} \le 
        C \|u\|_{\RL^2(\omega)}^{\alpha^{\operatorname{diam}(\BG)}} \|w\|^{1 -
        \alpha^{\operatorname{diam}(\BG)}}_{\RL^2(B_x(R') \times [-T,T])},
    \end{equation}
    and since the diameter of $\BG$ is bounded by the number of its vertices we
    obtain our claim after summing \eqref{eq:ineveryball} over all $y \in \BY$.
\end{proof}

Finally, we globalise Proposition \ref{prop: LM_largeballs}.

\begin{proposition}\label{prop:LM_global}
    Let $d \ge 2$ and $\kappa,R,\eps,T > 0$. There exist $C>0$ and $\beta \in
    (0,1)$ so that the following holds.

    For every $(M,g) \in \CM(d,\kappa)$, any pairs $u : M \to \R$ and $w : M
    \times [-T,T] \to \R$ solving
    \begin{equation}
        \begin{cases}
            (\Delta_g + \partial_t^2)w = 0 & \text{in } M \times [-T,T] \\
        w(y,0) \equiv 0  & \text{ and } \del_t w(y,0) = u(y) 
        \end{cases}
    \end{equation}
    and any open set $\omega \subset M$ such that for all $x \in M$,
    \begin{equation}
        \Vol(\omega \cap B_x(R)) \ge \eps \Vol(B_x(R))
    \end{equation}
    we have the estimate
	\[
	\|u\|_{\RL^2(M)}\leq C \|u\|_{\RL^2(\omega)}^{\beta}\|w\|_{\RL^2(M\times (-T,T))}^{1-\gamma}.
	\]
\end{proposition}
\begin{proof}
    Let $\BY$ be a maximal $R$-separated set inside $M$. By construction,
    \begin{equation}
        M \subset \bigcup_{y \in \BY} B_y(R),
    \end{equation}
    which imples, by Proposition \ref{prop: LM_largeballs}, that there is $C > 0$ and $\beta
    \in (0,1)$ so that
    \begin{equation}
        \label{eq:uboundedinballs}
	\|u\|_{\RL^2(M)}\leq \sum_{y\in \BY}\|u\|_{\RL^2(B_y(R))} \le C \sum_{y \in
    \BY}  \|u\|^\beta_{\RL^2(\omega \cap B_y(R))}
\|w\|^{1-\beta}_{\RL^2(B_y(R+r_\kappa) \times [-T,T])}
    \end{equation}
    There is $N > 0$, depending only on $d,\kappa$ and $R$, so that for every $y \in \BY$
    \begin{equation}
        \#\{z \in \BY : B_z(R) \cap B_y(R) \ne \varnothing\} \le N.
    \end{equation}
    Indeed, the intersection $B_z(R) \cap B_y(R)$ being non-empty is equivalent
    to $z \in B_y(2R)$ and, by Proposition \ref{prop:covering_balls}, there is
    $N$, independent of $y$, bounding the cardinality of an $R$-separated family
    in $B(y,2R)$. In particular,
    \begin{equation}
        \label{eq:uboundedmult}
	\sum_{y \in \BY}\|u\|_{\RL^2(\omega \cap B_y(R))}\leq N\|u\|_{\RL^2(\omega)}
\end{equation}
	and
    \begin{equation}
        \label{eq:wboundedmult}
        \sum_{y\in \BY}\|w\|_{\RL^2(B_y(R + \tfrac{r_\kappa}{4})\times
        (-T,T))}\leq N\|w\|_{\RL^2(M \times (-T,T))}.
    \end{equation}
    Inserting \eqref{eq:uboundedmult} and \eqref{eq:wboundedmult} in
    \eqref{eq:uboundedinballs} and applying H\"older's inequality, we obtain
    finally
    \begin{equation}
	\begin{aligned}
		\|u\|_{\RL^2(M)}
		&\leq C\sum_{y \in \BY}\|u\|_{\RL^2(\omega \cap
        B_y(R))}^{\beta}\|w\|_{\RL^2(B_y(R)\times (-T,T))}^{1-\beta}\\
		&\leq C\left(\sum_{y \in \BY}\|u\|_{\RL^2(\omega \cap B_y(R))}\right)^{\beta}\left(\sum_{y \in \BY}\|w\|_{\RL^2(B_y(R +
        r_\kappa) \times (-T,T))}\right)^{1-\beta}\\
		&\leq CN\|u\|_{\RL^2(\omega)}^{\beta}\|w\|_{\RL^2(M \times
        (-T,T))}^{1-\beta}.
	\end{aligned}
\end{equation}
	
\end{proof}

	\section{Proof of the main theorems}
\label{sec:endgame}

\begin{proof}[Proof of Theorem \ref{theo: main result}]
Let $M \in \CM(d,\kappa)$, $R, \eps > 0$ and let $\omega \subset M$ be an
$(\eps,R)$-thick set, and $u \in \RL^2(M)$. Recall that $\mu$ is the spectral
measure for  $\sqrt{-\Delta_g}$ and for $\Lambda > 0$, $\Pi_\Lambda$ is the spectral
projection on $[0,\Lambda]$, that is
\begin{equation}
    \Pi_\Lambda u = \int_0^{\Lambda}  \de \mu(\lambda) u.
\end{equation}
Define $w_\Lambda : M \times [-T,T]$ via the expression
\begin{equation}
    w_\Lambda(x,t) = \int_0^{\Lambda} \frac{\sinh(\lambda t)}{\lambda} [\de
    \mu(\lambda) u](x).
\end{equation}
By construction,
\[
(\Delta_g+\partial_t^2)w_{\Lambda}=0\qquad w_{\Lambda}|_{t=0}=0 \qquad
\partial_tw_{\Lambda}|_{t=0}=\Pi_{\Lambda}u,
\]
so that 
applying Proposition \ref{prop:LM_global}, we obtain
\begin{equation}
    \label{eq:afterLM}
\|\Pi_{\Lambda}u\|_{\RL^2(M)}\leq \|\Pi_{\Lambda}
u\|_{\RL^2(\omega)}^{\gamma}\|w_\Lambda\|_{\RL^2(M \times (-T,T))}^{1-\gamma}.
\end{equation}
For any bounded measurable function $\varphi \in
\mathrm{L}^\infty(\mathbb{R}^+)$, the functional calculus yields the estimate
\[
\|\varphi\left(\sqrt{-\Delta_g}\right) \Pi_\Lambda u\|_{\RL^2(M)} \leq \|\varphi\|_{\RL^\infty([0,\Lambda])} \|\Pi_\Lambda u\|_{\RL^2(M)}.
\]
Since $\|\sinh(\lambda \cdot)\|_{\RL^\infty(-T,T)} \le e^{\lambda T}$, we have that
\begin{equation}
    \|w_\Lambda\|_{\RL^2(M \times [-T,T])} \le e^{\Lambda T}
    \|\Pi_\Lambda u\|_{\RL^2(M)}.
\end{equation}
Inserting this into \eqref{eq:afterLM} gives us
\[
    \|\Pi_\Lambda u\|_{\RL^2(M)} \le C e^{\tfrac{T}{\gamma} \Lambda}
    \|u\|_{\RL^2(\omega)}, 
\]
in other words \eqref{eq: spec} holds for $\omega$.
\end{proof}

\begin{proof}[Proof of Theorem \ref{theo: fixed time obs}]
	Letting $v\in \RL^2(M)$, one has
\[
e^{s\Delta_g}v = \int_0^\infty e^{-s\lambda^2}\dd \mu(\lambda)v.
\]
Now, for $s>0$, if we let
\[
    w_s(x,t) = \int_0^\infty e^{-s\lambda^2}\frac{\sinh(\lambda t)}{\lambda}[\dd \mu(\lambda)v](x)
\]
one has again $w_s\in \RW^{1,2}(M\times [-T,T])$ with
\[
\|w_s\|_{\RL^2(M\times [-T,T])}\leq C(s,T)\|v\|_{\RL^2(M)};
\]
essentially, the quadratic frequency decay from the heat equation beats the growth of $\sinh$.

In particular, applying Proposition \ref{prop:LM_global}, we obtain
\[
\|e^{s\Delta_g}v\|_{\RL^2(M)}\leq C(s)\|e^{s\Delta_g}v\|_{\RL^2(\omega)}^{\gamma}\|v\|_{\RL^2(M)}^{1-\gamma},
\]
and the proof is complete.
\end{proof}

	\bibliographystyle{alpha}
	\bibliography{fullbib}
\end{document}